\documentclass[12pt,reqno]{amsart}
\usepackage{amsmath,amsthm,amssymb,mathrsfs,stmaryrd,color}
\usepackage[all]{xy}
\usepackage{url}
\usepackage{slashed}
\usepackage{geometry}

\usepackage[utf8]{inputenc}
\usepackage[T1]{fontenc}

\usepackage{relsize}
\usepackage[bbgreekl]{mathbbol}
\usepackage{amsfonts}

\usepackage{tikz-cd}

\usepackage[mathscr]{eucal}
\usepackage{eucal}
\usepackage{url}

\usepackage{setspace}
\usepackage{subcaption}
\usepackage{array}
\usepackage{wrapfig}
\usepackage{multirow}
\usepackage{tabularx}
\usepackage{pst-node}
\usepackage{mathtools}
\usepackage{bbm}

\usepackage{enumitem}
\setlist[enumerate]{itemsep=2pt,parsep=2pt,before={\parskip=2pt}}

\usepackage[colorlinks=true,hyperindex, linkcolor=magenta, pagebackref=false, citecolor=cyan,pdfpagelabels]{hyperref}

\newcommand{\cosimp}[3]{\xymatrix@R=50pt@C=50pt@1{#1 \ar@<.4ex>[r] \ar@<-.4ex>[r] & {\ }#2 \ar@<0.8ex>[r] \ar[r] \ar@<-.8ex>[r] & {\ } #3 \ar@<1.2ex>[r] \ar@<.4ex>[r] \ar@<-.4ex>[r] \ar@<-1.2ex>[r] & \cdots }}

\newcommand{\adjunction}[4]{\xymatrix@R=50pt@C=50pt@1{#1{\ } \ar@<0.3ex>[r]^-{ {\scriptstyle #2}} & {\ } #3 \ar@<0.3ex>[l]^{ {\scriptstyle #4}}}}

\newcommand{\defeq}{\vcentcolon=}

\newtheorem{theorem}{Theorem}[section]
\newtheorem*{theorem*}{Theorem}
\newtheorem*{definition*}{Definition}
\newtheorem{proposition}[theorem]{Proposition}
\newtheorem{lemma}[theorem]{Lemma}
\newtheorem{corollary}[theorem]{Corollary}

\theoremstyle{definition}

\newtheorem{definition}[theorem]{Definition}
\newtheorem{question}[theorem]{Question}
\newtheorem{remark}[theorem]{Remark}

\newtheorem{notation}[theorem]{Notation}

\DeclareMathOperator{\Stab}{Stab}
\DeclareMathOperator{\Frac}{Frac}

\DeclareMathOperator{\mmod}{mod}

\DeclareMathOperator{\Hilb}{Hilb}
\DeclareMathOperator{\SB}{SB}
\DeclareMathOperator{\cl}{cl}

\newcommand{\Vect}{\operatorname{Vect}}

\newcommand{\Spec}{\operatorname{Spec}}

\newcommand{\Id}{\operatorname{Id}}

\newcommand{\Gal}{\operatorname{Gal}}
\newcommand{\Sym}{\operatorname{Sym}}
\newcommand{\Gr}{\operatorname{Gr}}

\newcommand{\Bl}{\operatorname{Bl}}
\newcommand{\codim}{\operatorname{codim}}

\def \ZZ {\mathbb{Z}}

\def \CC {\mathbb{C}}
\def \AA {\mathbb{A}}
\def \NN {\mathbb{N}}

\def \kk {\Bbbk{}}
\def \LL {\mathbb{L}}

\def \PP {\mathbb{P}}
\def \ge {\geqslant}
\def \le {\leqslant}

\def \kk {\Bbbk{}}

\newcommand{\cE}{\mathcal{E}}
\newcommand{\cF}{\mathcal{F}}
\newcommand{\cG}{\mathcal{G}}
\newcommand{\cL}{\mathcal{L}}

\newcommand{\cV}{\mathcal{V}}
\newcommand{\cO}{\mathcal{O}}

\begin{document}

\title[Singularities of symmetric powers and irrationality of zeta functions]{Singularities of symmetric powers and irrationality of motivic zeta functions}
\author{Vladimir Shein}
\address{Department of Mathematics, Northwestern University, Evanston, IL 60208, USA}
\email{VladimirShein2029@u.northwestern.edu}

\begin{abstract}

Let $K_0(\mathcal{V}_{K})$ be the Grothendieck ring of varieties over a field $K$ of characteristic zero, and let $\mathbb{L} = [\mathbb{A}^1_K]$ denote the Lefschetz class. We prove that if a $K$-variety has $\mathbb{L}$-rational singularities, then all its symmetric powers also have $\mathbb{L}$-rational singularities. We then use this result to show that, for a smooth complex projective variety $X$ of dimension greater than one, the rationality of its Kapranov motivic zeta function $Z(X, t)$ (viewed as a formal power series over $K_0(\mathcal{V}_{\mathbb{C}})$) implies that the Kodaira dimension of $X$ is negative and that $X$ does not admit global nonzero differential forms of even degree. This extends the irrationality part of the Larsen–Lunts rationality criterion from the surface case to arbitrary dimension. We also discuss some applications of these results.
 
\end{abstract}

\maketitle


\section{Introduction}\label{section:introduction}

Throughout this text, $K$ denotes a field of characteristic zero. By a variety over $K$ we mean a reduced separated  scheme of finite type over $K$. For a point $x$ of a variety $X$ over $K$, we denote by $\kk(x)$ the residue field of $x$. If $X$ is smooth and projective, we denote its Kodaira dimension by $\kappa(X)$. 

Let $S$ be a variety over $K$, and let $\ZZ[\cV_S]$ be the free abelian group  generated by the isomorphism classes $[X]_S$ of  reduced separated $S$-schemes of finite type. By definition, the \textit{Grothendieck ring} $K_0(\cV_S)$ \textit{of varieties over} $S$ is the quotient of $\ZZ[\cV_S]$ by the subgroup generated by the elements
\[
[X]_S - [Z]_S - [X \, \backslash \, Z]_S
\]
for all pairs $(X, Z)$ such that $Z$ is a closed subscheme of $X$. The product in $K_0(\cV_S)$ is defined by $[X]_S \cdot [Y]_S = [(X \times_{S} Y)]_S$ on generators and extended by linearity. The class of $\AA_S^1$ in $K_0(\cV_S)$ is denoted by $\LL_S$. If $S = \Spec{K}$ (or if $S$ is understood), we simply write $[X]$ (resp. $\LL$) instead of $[X]_{S}$ (resp. $\LL_S$).

While the structure of $K_0(\cV_K)$ is in general poorly understood, there is a natural description for the quotient ring $K_0(\cV_K)/(\LL)$. Denote by $\SB/K$ the set of stably birational equivalence classes $\langle X \rangle$ of irreducible varieties over $K$ (recall that  $X, Y$ are called \textit{stably birational} if $X \times \PP^{n}$ is birational to $Y \times \PP^{m}$ for some $n, m \ge 0$). It is a commutative semigroup, with multiplication induced by the rule $\langle X \rangle \cdot \langle Y \rangle = \langle X \times Y \rangle$. Let $\ZZ[\SB/K]$ be the associated semigroup ring.

\begin{theorem}[{\cite{Lar_Lun}}\footnote{The original proof is over $\CC$ but it works over any field $K$ of characteristic zero because the main ingredient of the proof, the birational factorization theorem (\cite[Theorem 0.3.1, Remark 2]{AKMW}), holds over such $K$.}]  \label{Lar_Lun}
There is a unique ring isomorphism
\[
\Phi: K_0(\cV_K)/(\LL) \overset{\sim}{\longrightarrow} \ZZ[\SB/K]
\]
such that $\Phi([X]) = \langle X \rangle$ for every smooth connected projective variety $X$.
\end{theorem}

Theorem \ref{Lar_Lun} tells us that for a smooth connected projective variety $Y$, the image of its class $[Y]$ in $K_0(\cV_K)/(\LL)$ records the stable birational equivalence class of $Y$. However, if $Y$ is not smooth, then the geometric interpretation of its class in $K_0(\cV_K)/(\LL)$ is less clear.

For a quasi-projective variety $X$ over $K$, we denote by $\Sym^n(X)$ the \textit{$n$-fold symmetric power} of $X$. By definition, it is the quotient of $X^n$ by the permutation action of the symmetric group $S_n$. We adopt the convention that $\Sym^0(X) = \Spec{K}$.

The ring $K_0(\cV_K)$ has a natural $\lambda$-structure (in the sense of Def. \ref{subsection:lambda_rings}), defined by $\lambda^n([X]) = [\Sym^n(X)]$\footnote{As pointed out in \cite[Section 8]{LL2}, this $\lambda$-structure is \textit{not} special.}.
The \textit{Kapranov motivic zeta function} of $X$ is defined as the formal power series
\[
Z(X, t) \defeq \lambda_t([X])  = \sum_{n \ge 0} [\Sym^n(X)]t^n \in 1+tK_0(\cV_K)[[t]].
\]

M. Kapranov conjectured that it is natural to expect that $Z(X, t)$ is a rational function (as a formal power series over $K_0(\cV_K)$) for an arbitrary variety $X$ ~\mbox{\cite[Remark~1.3.5]{Kapranov}}. He showed that this is the case if $X$ is a curve with a rational point, and D. Litt later generalized this result to arbitrary geometrically connected curves ~\mbox{\cite[Theorem~10]{Litt2}}. However, it turned out that rationality fails in dimension two. It was proved by M. Larsen and V. Lunts ~\mbox{\cite[Theorem~1.1]{LL2}} that for a smooth complex projective surface $X$, the zeta function $Z(X, t)$ is rational if and only if $X$ has negative Kodaira dimension. Moreover, they proved irrationality in a stronger sense by showing that if $\kappa(X) \ge 0$, then $Z(X, t)$ is not \textit{pointwise} rational (see Def. \ref{def:rationality}). In this paper we prove the following generalization of this result to higher dimensions:

\begin{theorem} \label{irrationality_result}

Let $X$ be a smooth connected complex projective variety of dimension $d > 1$.  Assume that at least one of the following conditions hold:
\begin{itemize}
    \item[(a)] $\kappa(X) \ge 0$ (equivalently, there exists $n > 0$ such that $H^0(X, \omega_X^{\otimes n}) \neq 0$).
    \item[(b)] There exists $i > 0$ such that $H^0(X, \Omega_X^{2i}) \neq 0$.
\end{itemize}
Then $Z(X, t)$ is not pointwise rational.

\end{theorem}

\noindent This theorem thus shows that rationality of $Z(X, t)$ in fact imposes very strong conditions on the geometry of $X$.

Currently, the main tool for proving irrationality is to find a ring homomorphism $\mu: K_0(\cV_K) \to A$ for some ring $A$ (called a \textit{motivic measure}) such that the image of $Z(X, t)$ under $\mu$ is not rational as a formal power series in $A[[t]]$. We use the motivic measures
\[
\mu_k: K_0(\cV_K) \to \ZZ[M], 
\]
(see \ref{Measures} for the definition) constructed in ~\mbox{\cite{LL2}}, where $M$ denotes the multiplicative semigroup of polynomials $f \in \ZZ[s]$ with constant term $1$.

However, it is difficult to analyze the classes $\mu_k(\Sym^m(X))$ directly, because the measures $\mu_k$ have a natural geometric interpretation only for smooth connected projective  varieties, while $\Sym^m(X)$ is not smooth whenever $\dim(X) > 1$ and $m > 1$. In ~\mbox{\cite{LL2}}, this difficulty was overcome by working with the Hilbert scheme $\Hilb^m(X)$ of length $m$ subschemes of $X$, which is a resolution of singularities of $\Sym^m(X)$ when $\dim(X)=2$. However, if $\dim(X) > 2$, then $\Hilb^m(X)$ is also singular. For this reason, in this paper we work with an arbitrary resolution of singularities $\widetilde{\Sym}^m(X)$ of $\Sym^m(X)$ instead.

The key result that makes our arguments possible is the affirmative answer to part (a) of the following question of Larsen and Lunts:

\begin{question}\cite[Question 6.7]{LL2} \label{question} \begin{enumerate}
\item[(a)] Let $X$ be a smooth projective variety over $K$. Is it true that
\[
[\widetilde{\Sym}^m(X)] = [\Sym^m(X)]
\]
in $K_0(\cV_K)/({\LL})$? In other words, is the class $[\widetilde{\Sym}^m(X)] - [\Sym^m(X)]$ always divisible by $[\AA^1]$?

\item[(b)] More generally, let $Z$ be a smooth projective variety over $K$ with the action of a finite group $G$. Is it true that
\[
[\widetilde{Z/G}] = [Z/G]
\]
in $K_0(\cV_K)/({\LL})$, where $\widetilde{Z/G}$ is a resolution of singularities of $Z/G$?

\end{enumerate}

\end{question}

This question has been thoroughly studied in \cite{ES}. The authors made progress towards answering part (a), and also gave a negative answer to part (b), proving in particular that if $BG$ is not stably rational, then for any faithful $K$-linear representation $V$ of $G$ containing at least one copy of the trivial representation, one has $[\widetilde{\PP(V)/G}] \neq [\PP(V) / G]$ in $K_0(\cV_K)/({\LL})$.  In Section \ref{sec:singularities}, we show that the answer to Question \ref{question}(a) is positive, without using the results of \cite{ES}. Let us briefly explain our strategy.

We use the notion of $\LL$-rational singularities (Def. \ref{L_rat_fiber}), introduced in ~\mbox{\cite[Definition 4.2.4]{NS}}. As proved in \cite[Lemma 4.2.5]{NS}, if $Y$ has $\LL$-rational singularities, then $[\widetilde{Y}] = [Y]$ in $K_0(\cV_K)/({\LL})$ for any resolution of singularities $\widetilde{Y}$ of $Y$. We prove the following

\begin{theorem} \label{resolution}

Let $X$ be a quasi-projective variety over $K$. If $X$ has $\LL$-rational singularities, then  for every $n \in \ZZ_{>0}$ the $n$-th symmetric power $\Sym^n(X)$ has $\LL$-rational singularities.

\end{theorem}

This result, which may be of independent interest, in particular implies an affirmative answer to \ref{question}(a). We first establish a certain lemma about the Grothendieck ring of varieties (Lemma \ref{lambda_structure_on_R}), which allows us to use an inductive argument on $n$ to prove that $\Sym^n(\AA^k)$ has $\LL$-rational singularities for all $n, k \ge 0$ (Proposition \ref{main_proposition}). We then show that the general case of Theorem \ref{resolution} can be deduced from this.

After that, we prove Theorem \ref{irrationality_result} in Section \ref{section:main}. Our proof largely relies on the ideas of 
\cite{LL2}, though we generalize their strategy in several ways. As mentioned above, our main generalization is that we work with an arbitrary resolution $\widetilde{\Sym}^m(X)$. This is made possible mainly by Theorem \ref{resolution} and the result of D.Arapura and S. Archava, which states that singularities of ${\Sym}^m(X)$ are canonical ~\mbox{\cite[Proposition~1]{AA}}. More precisely, Theorem \ref{resolution} implies that $\mu_k(\Sym^m(X)) = \mu_k(\widetilde{\Sym}^m(X))$ for all $k \ge 1$, while ~\mbox{\cite[Proposition~1]{AA}} helps us study $\mu_k(\widetilde{\Sym}^m(X))$. We then show that, under the assumptions of Theorem \ref{irrationality_result}, one can find $k$ such that $\mu_k(Z(X, t))$ is not rational.

Finally, in Section \ref{sec:applications}, we discuss a couple of applications of the results from Sections \ref{sec:singularities} and \ref{section:main}. We show that Theorem \ref{resolution} together with the result of J. Koll\'{a}r allows us to easily compute the image of $Z(B, t)$ in $(K_0(\cV_K)/(\LL))[[t]]$ for a Severi-Brauer variety $B$ (Proposition \ref{ratinoalinquotient}). We then explain that, if $X$ satisfies the assumptions of Theorem \ref{irrationality_result}, one can use the measures $\mu_k$ to detect which symmetric powers of $X$ are not stably birational to each other (Proposition \ref{sym_are_not_stably_birational}), which generalizes {\cite[Theorem~19]{Litt}}. We also note that a positive answer to Question \ref{question}(a) implies that the motivic measure $\mu_1: K_0(\cV_K) \to \ZZ[M]$ is a homomorphism of $\lambda$-rings (Proposition \ref{hom_of_lambda}), which was already observed by Larsen and Lunts.

Let us also mention that, for the purposes of motivic integration, it is more natural to consider the localized ring $K_0(\cV_K)[1/\LL]$ instead of $K_0(\cV_K)$. Thus, one might ask when $Z(X, t)$ is rational as a formal power series over $K_0(\cV_K)[1/\LL]$. We do not investigate this interesting question in this paper, although we expect that the rationality of $Z(X, t)$ over $K_0(\cV_K)[1/\LL]$ should also impose strong restrictions on the geometry of $X$.

\textbf{Acknowledgements.} I am deeply grateful to Valery Lunts for a lot of fruitful discussions and valuable suggestions. I would also like to thank Michal Larsen for several helpful conversations and pieces of advice, and Sándor Kovács and Evgeny Shinder for answering my questions.

\vspace{0.3cm}

\section{Singularities of symmetric powers}\label{sec:singularities}

Following ~\mbox{\cite[Definition 4.2.4]{NS}} and ~\mbox{\cite[Definition 2.7]{ES}}, we introduce the notion of $\LL$-rational fibers and singularities. Our main result in this section is Theorem \ref{resolution} stating that symmetric powers of varieties with $\LL$-rational singularities again have $\LL$-rational singularities. We usually write $\pi_X: \widetilde{X} \to X$ for a resolution of singularities of $X$, in the sense that $\widetilde{X}$ is smooth and $\pi_X$ is proper, birational and an isomorphism away from the singular locus of $X$.

\begin{definition} \label{L_rat_fiber}

Let $X, Y$ be varieties over $K$. We say that a proper morphism $f: Y \to X$ has $\LL$-\textit{rational fibers} if for every (not necessarily closed) point $x$ of $X$ one has $[f^{-1}(x)] = 1$ in $K_0(\cV_{\kk(x)})/(\LL)$. We say that $X$ has $\LL$\textit{-rational singularities} if  $\pi_X: \widetilde{X} \to X$ has $\LL$-rational fibers for some resolution of singularities of $X$.

\end{definition}

\begin{remark} \label{independence}
    
It follows from the weak birational factorization theorem \cite[Theorem 0.3.1]{AKMW} that if $[\pi_X^{-1}(x)] = 1$ in $K_0(\cV_{\kk(x)})/(\LL)$ for some resolution of singularities $\pi_X: \widetilde{X} \to X$, then the same is true for any other resolution of singularities (\cite[Page 394]{NS}). In particular, the property of having $\LL${-rational singularities} does not depend on the choice of a resolution.
\end{remark}

\begin{lemma} \label{local}

Let $f: Y \to X$ be a proper morphism of varieties over $K$. The following are equivalent:

\begin{enumerate}

\item[1.] $f$ has $\LL$-rational fibers.

\item[2.] There exists a decomposition $X = \bigsqcup_{i = 1}^{r} W_i$ of $X$ into locally closed subsets such that $[f^{-1}(W_i)] = 1$ in $K_0(\cV_{W_i})/(\LL_{W_i})$ for every $i$.

\item[3.] There exists a decomposition $X = \bigsqcup_{i = 1}^{r} W_i$ of $X$ into locally closed subsets such that $f^{-1}(W_i) \to W_i$ has $\LL$-rational fibers.

\end{enumerate}

\end{lemma}

\begin{proof}

$\boxed{1. \Rightarrow 2.}$ Assume $f: Y \to X$ has $\LL$-rational fibers. If $z = \eta_Z$ is a generic point of some irreducible closed subset $Z \subset X$, then $\kk(z) = \varinjlim_{U \subset Z} \cO(U)$, where the limit is taken over all affine open subsets of $Z$. By spreading out (\cite[3.4]{NiSe}), it follows that $[f^{-1}(U)] \equiv [U] \mod{(\LL_{U})}$ in $K_0(\cV_{U})$ for some $U$. The desired decomposition into locally closed subsets now follows by noetherian induction.

$\boxed{2. \Rightarrow 3.}$ Clear (by functoriality of the Grothendieck ring).

$\boxed{3. \Rightarrow 1.}$ For any irreducible closed subset $Z \subset X$ with the generic point $z = \eta_{Z}$ there exists some $i$ such that $Z \cap W_i$ is open in $Z$ and nonempty. Hence, $z \in Z \cap W_i$, so $[f^{-1}(z)] = 1$ in $K_0(\cV_{\kk(z)})/(\LL)$ by functoriality of the Grothendieck ring.

\end{proof}

\begin{corollary}[{\cite[Lemma 4.2.5]{NS}}] \label{classes_coincide}
Let $f: Y \to X$ be a morphism of varieties that has $\LL$-rational fibers. Then $[Y] = [X]$ in $K_0(\cV_{K})/(\LL)$.

\end{corollary}

\begin{proof}

Using a decomposition $X = \bigsqcup_{i = 1}^{r} W_i$ from the $\boxed{1. \Rightarrow 2.}$ part of Lemma \ref{local}, we see that
\[
[Y] = \sum_{i=1}^r [f^{-1}(W_i)] \equiv \sum_{i=1}^r [W_i] = [X]
\]
in $K_0(\cV_{K})/(\LL)$.
    
\end{proof}

\begin{definition}
Let $S$ and $F$ be varieties over $K$. We say that a morphism $f: Y \to X$ of $S$-varieties is a piecewise trivial fibration with fiber $F$ if there exists a decomposition $X = \bigsqcup_{i=1}^r W_i$ of $X$ into locally closed subsets such that for every $i$, $f^{-1}(W_i)$ and $ F \times_{K} W_i$ are isomorphic as schemes over $W_i$.

\end{definition}

\begin{remark} \label{locally_trivial_fibration}

Let $f: Y \to X$ be a piecewise trivial fibration of $S$-varieties with fiber $F$. Then we have isomorphisms $f^{-1}(W_i) \overset{\sim}{\rightarrow} F \times_{K }W_i \overset{\sim}{\rightarrow} (F_S) \times_{S} W_i$ of varieties over $W_i$, and in particular over $S$. Therefore
\[
[Y]_S = \sum_{i=1}^r [f^{-1}(W_i)]_S  = \sum_{i=1}^r [F_S \times_S W_i]_S = [F_S] \cdot \sum_{i=1}^r [W_i]_S = [F_S] \cdot [X]_S
\]
in  $K_0(\cV_S)$. In particular, if $F \cong \AA^n$, then $[Y]_S = \LL^n_S [X]_S$ in $K_0(\cV_S)$.
    
\end{remark}

\begin{lemma} 
\label{etale_resolution}

Suppose that for every \textit{closed} point $x \in X$ there exists a variety $V_x$ that has $\LL$-rational singularities and an \'etale morphism $\psi: U_x \to V_x$ for some open neighborhood $U_x$ of $x$. Then $X$ has $\LL$-rational singularities.

\end{lemma}

\begin{proof}

Let $\pi: \widetilde{V}_x \to V_x$ and $\rho: \widetilde{X} \to X$ be resolutions of singularities, and let $\widetilde{U}_x = U_x \times_{V_x} \widetilde{V}_x$. Then $\widetilde{\pi} \defeq \psi^{\ast} (\pi): \widetilde{U}_x \to U_x$ is proper and birational as an \'etale pullback of $\pi$, and $\widetilde{U}_x$ is smooth, so $\widetilde{\pi}$ is also a resolution of singularities. Since the class of morphisms with $\LL$-rational fibers is closed under base change (\mbox{\cite[Lemma 2.9]{ES}}), we conclude that $U_x$ has $\LL$-rational singularities.

Now let $j: U_x \hookrightarrow X$ and $\widetilde{U}'_x = U_x \times_X \widetilde{X}$. Analogously, $\widetilde{\rho} \defeq j^{\ast} (\rho): \widetilde{U}'_x \to U_x$ is a resolution of singularities, so $\widetilde{\rho}$ has $\LL$-rational fibers. The lemma now follows because $X = \bigcup_{x \in X_{ \cl}} U_x$.
\end{proof}

Before proceeding further, let us recall the definition of a $\lambda$-ring following \cite[Sec. 4]{Gr}.

\begin{definition} \label{subsection:lambda_rings}

A $\lambda$-structure on a ring $A$ is a sequence of maps $(\lambda^i: A \to A)_{i \ge 0}$ (called $\lambda$-operations) such that
\begin{equation} \label{eq:lambda}
\aligned
\lambda^0(x) & = 1, \\
\lambda^1(x) & = x, \\
\lambda^n(x+y) &= \sum_{i+j = n} \lambda^i(x)\lambda^j(y) \quad \text{for } n \geqslant 0.
\endaligned
\end{equation}

A $\lambda$-ring is a ring with a $\lambda$-structure. A $\lambda$-homomorphism is a homomorphism of rings commuting with all $\lambda$-operations. It is clear that $(\lambda^i: A \to A)_{i \ge 0}$  define a $\lambda$-structure on $A$ if and only if the map
\[
\lambda_t: (A, \, +) \to (1+ tA[[t]], \, \times ), \quad a \mapsto \sum_{n \geqslant 0} \lambda^n(a)t^n
\]
is a homomorphism of groups.

For every ring $A$, one can equip $1+ tA[[t]]$ with a canonical $\lambda$-structure. If $A$ is in addition a $\lambda$-ring, then one says that $A$ is a \textit{special} $\lambda$-ring if $\lambda_t$ is a $\lambda$-homomorphism.


An ideal $I \subset A$ of a $\lambda$-ring $A$ is called a $\lambda$-ideal if $\lambda^i(t) \in I$ for all $t \in I$ and all $i \ge 1$. For a $\lambda$-ideal $I \subset A$, the $\lambda$-structure on $A$ canonically descends to $A/I$, making it a $\lambda$-ring.

\end{definition}

\begin{proposition} \label{sym_have_rational_fibers}

\begin{enumerate}

\item[(a)] Assume that $f: X \to S$ and $f': X' \to S'$ are morphisms of $K$-varieties which have $\LL$-rational fibers. Then the product morphism $f \times f': X \times X' \to S \times S'$ has $\LL$-rational fibers.

\item[(b)] Assume that $f: X \to S$ has $\LL$-rational fibers. Then the $n$-th symmetric power map $\Sym^n(f): \Sym^n(X) \to \Sym^n(S)$ has $\LL$-rational fibers for every $n \ge 0$.

\end{enumerate}

\end{proposition}

As we will see below, the only non-trivial part of this proposition is $(b)$. To prove this, let us introduce the graded ring
\[
R_S \defeq \bigoplus_{i \ge 0} K_0(\cV_{\Sym^i(S)}).
\]
For classes $[X] \in K_0(\cV_{\Sym^i(S)})$, $[Y] \in K_0(\cV_{\Sym^j(S)})$ one can naturally view their product $[X \times Y]$ as an element in $K_0(\cV_{\Sym^{i+j}(S)})$ via the composition
\[
X \times Y \to \Sym^i(S) \times \Sym^j(S) \to \Sym^{i+j}(S).
\]
Thus, extended by linearity, it gives a well-defined graded ring structure on $R_S$.

Consider a homogeneous ideal
\[
(\LL_{R}) \defeq \bigoplus_{i \ge 0} (\LL_{\Sym^{i}(S)}) \subset \bigoplus_{i \ge 0} K_0(\cV_{\Sym^i(S)}).
\]

\begin{lemma} \label{lambda_structure_on_R}

The ring $R_S$ admits a unique $\lambda$-structure such that for any variety $Y$ over $\Sym^{i}(S)$ one has
\begin{equation} \label{lambda_strucure}
\lambda^n([Y]) = [\Sym^{n}(Y)] \in K_0(\cV_{\Sym^{in}(S)}) \subset R_S,
\end{equation}
where $\Sym^{n}(Y)$ is viewed as a variety over $\Sym^{in}(S)$ via the composition
\[
\Sym^n(Y) \to \Sym^n(\Sym^{i}(S)) \to \Sym^{in}(S).
\]
Moreover, $(\LL_R)$ is a $\lambda$-ideal (thus, this $\lambda$-structure descends to a $\lambda$-structure on $R_S/(\LL_R)$).
    
\end{lemma}

\begin{proof}

It is straightforward to check that there are unique maps
\[
\lambda^n: \bigoplus_{i \ge 0} \ZZ[\cV_{\Sym^i(S)}] \to \bigoplus_{i \ge 0} K_0(\cV_{\Sym^i(S)})
\]
satisfying (\ref{eq:lambda}) and such that for a variety $Y$ over $\Sym^i(S)$ one has $\lambda^n(Y) = [\Sym^n(Y)] \in K_0(\cV_{\Sym^{in}(S)})$. To see that the maps $\lambda^n$ factor through $R_S$, it is enough to note that if $X$ is a variety over $\Sym^k(S)$, $Z \subset X$ is a closed subvariety  and $U = X \, \backslash \, Z$, then there exists a decomposition
\[
\Sym^{n}(X) = \bigsqcup_{i=0}^{n} (\Sym^{i}(U) \times \Sym^{n-i}(Z))
\]
into locally closed subvarieties and it respects the structure maps to $\Sym^{kn}(S)$.

Finally, to prove that $(\LL_R)$ is a $\lambda$-ideal, it is enough to show that $\lambda^n(x) \in (\LL_R)$ for all $n \ge 1$ and all additive generators $x$ of $(\LL_R)$. Thus, we need to prove that for any variety $X$ over $\Sym^i(S)$ one has
\[
[\Sym^n(X \times \AA^1_K)] \subset (\LL_{\Sym^{in}(S)}).
\]
It follows from the proof of \cite[Lemma 4.4]{Got} that for every $n \ge 0$ the natural projection map
\[
p_1: \Sym^{n}( X \times \AA^1_K) \to \Sym^{n}(X)
\]
is a piecewise trivial fibration over $\Sym^n(X)$ with fiber $\AA^n$. Therefore 
\[
[\Sym^{n}( X \times \AA^1_K)] = [\Sym^{n}(X)] \cdot \LL^n_{\Sym^n(X)} 
\]
in $K_0(\cV_{\Sym^n(X)})$ by Remark \ref{locally_trivial_fibration}. Since $X$ was over $\Sym^i(S)$, we deduce that this equality also holds in $K_0(\cV_{\Sym^{in}(S)})$.

\end{proof}

\begin{corollary}
    
\label{unit_is_preserved}

\begin{enumerate}

\item[(a)] Let $S$ and $S'$ be varieties over $K$, and let $X$ and $X'$ be varieties over $S$ and $S'$ respectively such that $[X] = 1 $ in $K_0(\cV_{S})/(\LL_{S})$ and $[X'] = 1$ in $K_0(\cV_{S'})/(\LL_{S'})$. Then $[X \times X'] = 1$ in $K_0(\cV_{S \times S'})/(\LL_{S \times S'})$.

\item[(b)] Let $S$ be a variety over $K$, and let $X$ be a variety over $S$ such that $[X] = 1$ in $K_0(\cV_{S})/(\LL_{S})$. Then $[\Sym^{n}(X)] = 1$ in $K_0(\cV_{\Sym^n(S)})/(\LL_{\Sym^n(S)})$\footnote{Here, if $f: X \to S$ is a structure map, then we assume that $\Sym^n(X)$ is naturally a variety over $\Sym^n(S)$ via the map $\Sym^n(f): \Sym^n(X) \to \Sym^n(S)$.}.

\end{enumerate}
    
\end{corollary}

\begin{proof}

For part (a), note that there is a natural homomorphism of rings:
\[
p_{S \times S'}: K_0(\cV_{S})/(\LL_{S}) \otimes_K K_0(\cV_{S'})/(\LL_{S'}) \to K_0(\cV_{S \times S'})/(\LL_{S \times S'}), \quad [X] \otimes [X'] \to [X \times X'].
\]
In particular, $p_{S \times S'}(1 \otimes 1) = 1$, so part (a) follows.

Part (b) follows immediately from Lemma \ref{lambda_structure_on_R}. Indeed, we have
\[
R_S / (\LL_R) \cong \bigoplus_{i \ge 0} K_0(\cV_{\Sym^i(S)})/(\LL_{\Sym^i(S)}),
\]
and $\lambda^n$ sends $K_0(\cV_{S})/(\LL_{S})$ to $K_0(\cV_{\Sym^n(S)})/(\LL_{\Sym^n(S)})$.

\end{proof}

\begin{proof}[Proof of Proposition \ref{sym_have_rational_fibers}]

We will prove part $(b)$. The reader can check that the proof of part $(a)$ is completely analogous.

By Lemma \ref{local}, there is a decomposition $S = \bigsqcup_{i = 1}^{r} W_i$ of $S$ into locally closed subsets such that $[f^{-1}(W_i)] = 1 $ in $K_0(\cV_{W_i})/(\LL_{W_i})$ for every $i$. It gives a decomposition
\[
\Sym^n(S) = \bigsqcup_{i_1+ \ldots i_r = n} \Sym^{i_1}(W_1) \times \cdots \times \Sym^{i_r}(W_r),
\]
where each $\Sym^{i_1}(W_1) \times \cdots \times \Sym^{i_r}(W_r)$ is again locally closed in $\Sym^{n}(S)$. We have
\[
[(\Sym^n(f))^{-1}(\Sym^{i_1}(W_1) \times \cdots \times \Sym^{i_r}(W_r))] = 
[\Sym^{i_1}(f^{-1}(W_1)) \times \cdots \times \Sym^{i_r}(f^{-1}(W_r))]
\]
and by Corollary \ref{unit_is_preserved} this class is equal to $1$ in $K_0(\cV_{\Sym^{i_1}(W_1) \times \cdots \times \Sym^{i_r}(W_r)}) / (\LL)$. By Lemma \ref{local} again, we conclude that $\Sym^n(f)$ has $\LL$-rational fibers.

\end{proof}

Before proving Theorem \ref{resolution}, let us fix some notation. Recall that it follows from \cite[Theorem 3.36]{Kollar2} that over a field of characteristic zero there exists a functorial resolution of singularities, i. e. a functor $F: X \mapsto (\pi_X: \widetilde{X} \to X)$ from the category of varieties over $K$ and \textit{smooth} morphisms to the category of smooth varieties and smooth morphisms such that $\pi_X: \widetilde{X} \to X$ is a resolution of singularities, and for any smooth $f: X \to Y$ the canonical map $\widetilde{X} \to X \times_{Y} \widetilde{Y}$ is an isomorphism. From now on in this section we assume that we have fixed a functorial resolution of singularities.

\begin{notation} \label{decompositon_of_sym}
    
We will use the following natural decomposition of symmetric powers into locally closed subsets. Let $(X^{r})^{\circ}$ be the complement in $X^r$ of the union of the big diagonals in $X^r$. Given positive integers $d_1 \le d_2 \le \cdots \le d_r$ such that $d_1 + \cdots + d_r = n$, there is a locally closed embedding $(X^{r})^{\circ} \hookrightarrow X^n$ given by $\Delta_{d_1} \times \ldots \times \Delta_{d_r}$, where $\Delta_{i}: X \to X^{i}$ are the diagonal embeddings. Let $\lambda$ be the partition of $n$ given by $d_1, \ldots, d_r$. We denote by $S_{\lambda}(X)$ (or simply $S_{\lambda}$ when $X$ is understood) the image of $(X^{r})^{\circ}$  in $X^{n}$, and by $X_{\lambda}$ the image of $S_{\lambda}(X)$ in $\Sym^n(X)$ under the natural projection $\pi: X^n \to \Sym^n(X)$. It is easily seen that the $X_{\lambda}$ give a partition of $\Sym^n(X)$ into locally closed subvarieties.

We will write $\Delta^n_X \defeq X_{(n^1)} \subset \Sym^n(X)$ for the small diagonal in $\Sym^n(X)$. We also denote by $\widetilde{\Sym}^n(X)$ a functorial resolution of singularities of the symmetric power $\Sym^n(X)$. 
\end{notation}

We deduce Theorem \ref{resolution} from the following main proposition:
 
\begin{proposition} \label{main_proposition}

For any $n, k \in \ZZ_{>0}$ the symmetric powers of affine spaces $\Sym^n(\AA^k)$ have $\LL$-rational singularities.
    
\end{proposition}

\begin{proof}

We will use the notation $\lambda = (1^{l_1}, \ldots, n^{l_n})$ for a partition of $n$ which corresponds to a decomposition $n = d_1 + \cdots + d_r$ such that the first $l_1$ of the $d_i$ are equal to $1$, the next $l_2$ of the $d_i$ are equal to $2$ and so on. Let $S_{\lambda} = S_{\lambda} (\AA^k)$ be as in Notation \ref{decompositon_of_sym}, and let $H_{\lambda} = \{g \in S_n \mid gS_{\lambda} = S_{\lambda} \}$ be the normalizer of $S_{\lambda}$. We have $H_{\lambda} = H_1 \rtimes H_2$, where $H_1 = \prod_{i=1}^r S_{d_i}$ and $H_2 = \prod_{j} S_{l_j}$, where the second product is over all $j \in \{1, \ldots, n \}$ such that $l_j \neq 0$.  Note that since $H_1$ acts trivially on $S_{\lambda}$, we have $\AA^k_{\lambda} \cong S_{\lambda}/H_{\lambda} \cong (S_{\lambda}/H_1)/H_2 \cong S_{\lambda}/H_2$.

Let us fix $k$. We break the proof into three steps.

\vspace{0.4cm}

\textit{Step 1. Let $\rho_n: \widetilde{\Sym}^n(\AA^k) \to \Sym^n(\AA^k)$ be a functorial resolution of singularities. Then  $\rho^{-1}_n(\Delta^n_{\AA^k}) \overset{\rho_n}{\longrightarrow} \Delta^n_{\AA^k}$ is a trivial fibration, i. e. there exists an isomorphism $\rho^{-1}_n(\Delta^n_{\AA^k}) \overset{\sim}{\to} \Delta^n_{\AA^k} \times F_n$ of varieties over $\Delta^n_{\AA^k}$ for some variety $F_n$}\footnote{In fact, a similar argument (with slight modifications) proves that for a functorial resolution of singularities $\rho_n$ and any $S_{\lambda}$ such that $H_2 = e$ the restriction $\rho_n^{-1}(S_{\lambda}) \to S_{\lambda}$ is a trivial fibration.}.

\vspace{0.3cm}

Let $N  = \{(x_1, \ldots x_n) \in (\AA^{k})^{n} \mid x_1 + \cdots + x_n = 0  \} \subset \AA^{kn}$ be a linear subspace complementary to the small diagonal $S_{(n^1)} \subset \AA^{kn}$. It is preserved by the action 
 of $S_n$ and we have an $S_n$-equivariant isomorphism $ N \times S_{(n^1)} \overset{\sim}{\longrightarrow} (\AA^{k})^{n}$ given by $(x, y) \mapsto x+y$.
Since the action of $S_n$ on $S_{(n^1)}$ is trivial, it induces an isomorphism $N / S_n \times \Delta^n_{\AA^k} \overset{\sim}{\longrightarrow} \Sym^{n}(\AA^k)$.
Since $\Delta^n_{\AA^k} \cong \AA^k$ is smooth, we have a Cartesian square
\begin{equation*}
\begin{tikzcd}[row sep=1cm]
\widetilde{N / S_n} \times \Delta^n_{\AA^k} \arrow{d}{\pi_{n} \times \Id} \arrow{r}{\sim} & \widetilde{\Sym}^n(\AA^k) \arrow{d}{\rho_n} \\
N / S_n \times \Delta^n_{\AA^k} \arrow{r}{\sim}& \Sym^{n}(\AA^k), 
\end{tikzcd}
\end{equation*}
where $\pi_n: \widetilde{N / S_n} \to N/S_n$ is a functorial resolution. In particular, $\rho^{-1}_n(\Delta^n_{\AA^k}) \overset{\sim}{\rightarrow} \pi_n^{-1}(0) \times \Delta^n_{\AA^k}$.

\vspace{0.4cm}

\textit{Step 2. Let $U_n  = \Sym^n(\AA^k) \, \backslash \, \Delta^n_{\AA^k}$ be the complement of the small diagonal. Assume that $\rho_i: \widetilde{\Sym}^i(\AA^k) \to \Sym^i(\AA^k)$ has $\LL$-rational fibers for all $i < n$. Then $ \rho^{-1}_n(U_n) \overset{\rho_n}{\longrightarrow} U_n$ has $\LL$-rational fibers.}

\vspace{0.3cm}

By Lemma \ref{local}, it is enough to prove that $\rho^{-1}_n(\AA^k_{\lambda}) \to \AA^k_{\lambda}$ has $\LL$-rational fibers for every partition $\lambda \neq (n^1)$. Note that
\[
\AA^{kn}/H_{\lambda} \cong \prod_{i=1}^n \Sym^{l_i} (\Sym^i(\AA^k)),
\]
where $l_n = 0$ if $\lambda \neq (n^1)$. Consider the composition:
\[
\prod_{i=1}^{n-1} \widetilde{\Sym}^{l_i} (\widetilde{\Sym}^i(\AA^k)) \overset{\alpha}{\longrightarrow} \prod_{i=1}^{n-1} {\Sym}^{l_i} (\widetilde{\Sym}^i(\AA^k)) \overset{\beta}{\longrightarrow} \prod_{i=1}^{n-1} \Sym^{l_i} (\Sym^i(\AA^k)) 
\]
where the map $\beta$ is the product of symmetric powers of the maps $\rho_i$, and the map $\alpha$ is the product of resolutions of singularities of ${\Sym}^{l_i} (\widetilde{\Sym}^i(\AA^k))$. It is clear that $\alpha$ and $\beta$ are both proper and birational, so $\beta \circ \alpha$ is a resolution of singularities (although it might not be functorial).

For every point $x \in S_{\lambda}$ we have $\, \Stab_{S_n}(x) \subset H_{\lambda}$, so the quotient morphism $\mu: \AA^{kn} / H_{\lambda} \to \Sym^{n}(\AA^k)$ is \'etale at $y$ for every point $y \in S_{\lambda}/H_{\lambda}$ by \cite[Expos\'e V. Proposition 2.2]{SGA1}. Therefore, there exist an open subset $W' \subset \AA^{kn}/H_{\lambda}$ containing $S_{\lambda}/H_{\lambda}$, such that $\left. \mu \right|_{W'}$ is \'etale. Moreover, by shrinking $W'$ if needed we can assume that $W' \cap \mu^{-1}(\AA^k_{\lambda}) = S_{\lambda}/H_{\lambda}$.

Let $W$ be the preimage of $W'$ under the projection $\AA^{kn} \to \AA^{kn}/H_{\lambda}$, so that $W' = W/H_{\lambda}$. Consider a commutative diagram
\begin{equation*} \label{diag2}
\begin{tikzcd}[row sep=1cm]
\prod_{i=1}^{n-1} \widetilde{\Sym}^{l_i} (\widetilde{\Sym}^i(\AA^k)) \arrow{d}{\beta \circ \alpha}
& \widetilde{{W}/{H_{\lambda}}} \arrow{r}{\widetilde{\mu}} \arrow{d}{\pi_{W/H_{\lambda}}}
& \widetilde{\Sym}^{n}(\AA^k) \arrow{d}{\rho_n}
\\
\prod_{i=1}^{n-1} \Sym^{l_i} (\Sym^i(\AA^k))  \arrow[r, hookleftarrow]
& {W}/{H_{\lambda}}  \arrow{r}{\left. \mu \right|_{W'}}
& \Sym^{n}(\AA^k)
\\
{S_{\lambda}}/{H_{\lambda}} \arrow[hookrightarrow]{u} 
& {S_{\lambda}}/{H_{\lambda}} \arrow[hookrightarrow]{u} \arrow{r}{\sim} \arrow[l, "\sim" above] 
& \AA^k_{\lambda} \arrow[hookrightarrow]{u}
\end{tikzcd}
\end{equation*}
where $\pi_{W/H_{\lambda}}$ is a functorial resolution. Since $\left. \mu \right|_{W'}$ is \'etale and  $(\left. \mu \right|_{W'})^{-1}(\AA^k_{\lambda}) = S_{\lambda}/H_{\lambda}$, it follows that all squares of the diagram are Cartesian and $\rho_n^{-1}(\AA^k_{\lambda}) \cong \pi^{-1}_{W/H_{\lambda}}(S_{\lambda}/H_{\lambda})$.

Note that by Remark \ref{independence} applied to $W/H_{\lambda}$, the map $\pi^{-1}_{W/H_{\lambda}} (S_{\lambda}/H_{\lambda}) \to S_{\lambda}/H_{\lambda}$ has $\LL$-rational fibers if and only if $({\beta \circ \alpha})^{-1} (S_{\lambda}/H_{\lambda}) \to S_{\lambda}/H_{\lambda}$ has $\LL$-rational fibers. Since $\rho_i: \widetilde{\Sym}^i(\AA^k) \to \Sym^i(\AA^k)$ has $\LL$-rational fibers for all $i < n$ by assumption, it follows from Proposition \ref{sym_have_rational_fibers} that $\beta$ has $\LL$-rational fibers as well. It remains to note that $\beta^{-1}(S_{\lambda}/H_{\lambda})$ lies in the smooth locus of $\prod_{i=1}^{n-1} {\Sym}^{l_i} (\widetilde{\Sym}^i(\AA^k))$, therefore $\alpha^{-1}(\beta^{-1}(S_{\lambda}/H_{\lambda})) \cong \beta^{-1}(S_{\lambda}/H_{\lambda})$. Hence, we conclude that $({\beta \circ \alpha})^{-1} (S_{\lambda}/H_{\lambda}) \to S_{\lambda}/H_{\lambda}$ has $\LL$-rational fibers.

\vspace{0.4cm}

\textit{Step 3. $\rho_n: \widetilde{\Sym}^n(\AA^k) \to \Sym^n(\AA^k)$ has $\LL$-rational fibers for every $k, n > 0$.}

\vspace{0.3cm}

We prove the statement by induction on $n$ (for every fixed $k$). If $n = 1$, then there is nothing to prove. Assume now that $\rho_i$ has $\LL$-rational fibers for every $i < n$. By step 2, $\rho_n^{-1}(U_n) \to U_n$ has $\LL$-rational fibers, so by Lemma \ref{local} it is enough to prove that $\rho_n^{-1}(\Delta^n_{\AA^k}) \to \Delta^n_{\AA^k}$ has $\LL$-rational fibers. By step 1, $\rho^{-1}_n(\Delta^n_{\AA^k}) \overset{\sim}{\rightarrow} \pi_n^{-1}(0) \times \Delta^n_{\AA^k}$, hence we only need to prove that $[\pi_n^{-1}(0)] = 1$ in $K_0(\cV_{K})/(\LL)$. We will do this by exploiting a functorial resolution of singularities
\[
\rho'_n: \widetilde{\Sym}^n(\PP^k) \to {\Sym}^n(\PP^k)
\]
of symmetric powers of projective spaces.

Note that $\Sym^n({\PP}^k)$ can be covered by open subsets isomorphic to $\Sym^n({\AA}^k)$ and such that $ {\Sym}^n(\AA^k) \cap \PP^k_{\lambda} = \AA^k_{\lambda} \subset \PP^k_{\lambda}$ for all partitions $\lambda$. Indeed, let us assume first that $K = \overline{K}$. In this case, for any closed point $p = (p_1, \ldots, p_n) \in {\Sym}^n(\PP^k)$ we can find a hyperplane $P \subset \PP^k$ such that $p_i \notin P$ for all $i$. If we denote by $\iota_{P}: \AA^k \cong (\PP^k \, \backslash \, P) \hookrightarrow \PP^k$ the corresponding inclusion, then the image of $\iota_{P}$ is an open subvariety of $\PP^k$, isomorphic to $\AA^k$ and containing $p_i$ for all $i$. In particular, $\Sym^n(\iota_{P}): \Sym^{n}(\AA^k) \hookrightarrow {\Sym}^n(\PP^k)$ is an open neighborhood of $p$, and all such neighborhoods form the desired cover. Since the above construction is $\Gal(\overline{K}/K)$-invariant, we see that the statement is true over $K$ as well.

Let ${U}'_n = {\Sym}^n(\PP^k) \, \backslash \, \Delta^n_{\PP^k}$. The observation from the previous paragraph shows that ${U}'_n$ can be covered by open subsets isomorphic to $U_n$, thus $(\rho'_n)^{-1}({U}'_n) \to {U}'_n$ has $\LL$-rational fibers. By the same observation, $\rho'_n$ is a trivial fibration over ${\Sym}^n(\AA^k) \cap \Delta^n_{\PP^k} =\Delta^n_{\AA^k} \subset \Delta^n_{\PP^k}$ with fiber $\pi_n^{-1}(0)$ (for each open subset  ${\Sym}^n(\AA^k)$ from the cover), so $\rho'_n$ is a piecewise trivial fibration over $\Delta^n_{\PP^k}$ with fiber $\pi_n^{-1}(0)$. Therefore, we have the following series of equalities in $K_0(\cV_K)/(\LL)$:
\begin{align*}
[\widetilde{\Sym}^n(\PP^k)]
& = [(\rho'_n)^{-1}(U'_n)]+[(\rho'_n)^{-1}( \Delta^n_{\PP^k})] \equiv [U'_n] + [\pi_n^{-1}(0)] \cdot [\Delta^n_{\PP^k}] \\
    & = ([\Sym^{n}(\PP^k)]- [\PP^k]) + [\pi_n^{-1}(0)] \cdot [\PP^k] \equiv [\pi_n^{-1}(0)] ,   
\end{align*}
where in the last equality we used that $[\Sym^{n}(\PP^k)] = [\PP^k] =  1$ in $K_0(\cV_K)/(\LL)$.

Since $\Sym^n(\PP^k)$ is rational (\cite[Ch. 4. Theorem 2.8]{Discriminants}), it follows that $\widetilde{\Sym}^n(\PP^k)$ is also rational. Additionally,  $\widetilde{\Sym}^n(\PP^k)$ is smooth and projective, so by Theorem \ref{Lar_Lun} we have $ [\widetilde{\Sym}^n(\PP^k)] = 1$ in $K_0(\cV_K)/(\LL)$, and thus $[\pi_n^{-1}(0)] = 1$ in $K_0(\cV_K)/(\LL)$.

\end{proof}

To deduce Theorem \ref{resolution}, we will use a lemma by D. Luna.

\begin{lemma}[{\cite[Lemme Fondamental]{Lu1}}] \label{Fondamental}

Let $G$ be a reductive group. Let $Y, Z$ be affine varieties over $K$ with a $G$-action, and let $\varphi: Y \to Z$ be a $G$-equivariant morphism. Let $y \in Y$ be a point such that $\varphi$ is \'etale at $y$, $Y$ is normal at $y$, $Z$ is normal at $\varphi(y)$, the orbits $Gy$ and $G(\varphi(y))$ are closed in $Y$ and $Z$ respectively, and $\Stab_{G}(y) = \Stab_{G}({\varphi(y)})$. Then there exists a $G$-invariant affine open neighborhood $W$ of $y$ such that the morphism $\varphi/G: W/G \to Z/G$ is \'etale. 

\end{lemma}

\begin{proof}[Proof of Theorem \ref{resolution}] First we assume that $X$ is smooth. Let $x \in X^n$ be a closed point given by a tuple $ (x_1, \ldots, x_n) \in (X(\overline{K}))^n$. Since $X$ is quasi-projective, there exists an affine open subvariety $V \subset X$ such that $x_i \in V(\overline{K})$ for all $i$. I claim that there exists a morphism $f: V \to \AA^k$ such that $f_{\overline{K}}: V_{\overline{K}} \to \AA^k_{\overline{K}}$ is \'etale at $x_i$ for all $i$ and $f_{\overline{K}}(x_i) \neq f_{\overline{K}}(x_j)$ for $x_i \neq x_j$. Indeed, without loss of generality we can assume that $V$ is a closed $k$-dimensional subvariety of an affine space $\AA^m$. The space of orthogonal projections $V_{\overline{K}} \to \AA^k_{\overline{K}}$ to linear $k$-dimensional subspaces is parametrized by the Grassmannian $(\Gr_{k, m})_{\overline{K}}$, and each of the conditions that $f_{\overline{K}}$ is \'etale at $x_i$ and $f_{\overline{K}}(x_i) \neq f_{\overline{K}}(x_j)$ corresponds to a nonempty irreducible open subset of $(\Gr_{k, m})_{\overline{K}}$. Moreover, the set of $K$-rational points of $\Gr_{k, m}$ is dense, therefore a projection which is defined over $K$ and with the desired properties  exists.

Since the set of points where a map is \'etale is open, there is an open open subset $U \subset V$ such that $x_i \in U(\overline{K})$ for all $i$ and $f_{\overline{K}}: U_{\overline{K}} \to \AA^k_{\overline{K}}$ is \'etale. By shrinking $U$ if needed, we can assume that it is affine. Note that the open subset $U^n \subset X^n$ contains $h(x)$ for all $h \in S_n$ and the morphism $f^n: U^n \to (\AA^k)^n$ is also \'etale (since its base change to $\overline{K}$ is \'etale as a product of \'etale morphisms). Moreover, one has $\Stab_{S_n}(x) = \Stab_{S_n}(f^n(x))$ because $f_{\overline{K}}(x_i) \neq f_{\overline{K}}(x_j)$ for $x_i \neq x_j$. Hence, by Lemma \ref{Fondamental}, there exists an $S_n$-invariant affine open neighborhood $W \subset U^n$ of $x$ such that $\Sym^n(f): W/S_n \to \Sym^n(\AA^k)$ is \'etale. The statement of the theorem now follows from Lemma \ref{etale_resolution} and Proposition \ref{main_proposition}.

Finally, assume that $X$ has $\LL$-rational singularities, and let $\pi_X: \widetilde{X} \to X$ be a resolution of singularities. Consider the composition:
\[
\widetilde{\Sym}^n(\widetilde{X}) \overset{\chi_n}{\longrightarrow} {\Sym}^n(\widetilde{X}) \overset{\Sym^n(\pi_X)}{\longrightarrow} {\Sym}^n({X}),
\]
where $\chi_n$ is a resolution of singularities of ${\Sym}^n(\widetilde{X})$. Since $\widetilde{\Sym}^n(\widetilde{X})$ is smooth and both $\chi_n$ and $\Sym^n(\pi_X)$ are proper and birational, their composition is a resolution of singularities of ${\Sym}^n({X})$. Now, $\Sym^n(\pi_X)$ has $\LL$-rational fibers by Proposition \ref{sym_have_rational_fibers}, and $\chi_n$ has $\LL$-rational fibers by the smooth case. Hence, their composition $\Sym^n(\pi_X) \circ \chi_n$ has $\LL$-rational fibers (by \cite[Lemma 2.9]{ES}).

\end{proof}

\begin{corollary} \label{class_of_resolution_coincide}

Let $X$ be a quasi-projective variety over $K$ which has $\LL$-rational singularities, and let $\widetilde{\Sym}^n(X)$ be a resolution of singularities of $\Sym^n(X)$. Then
\[
[\widetilde{\Sym}^n(X)] = [\Sym^n(X)]
\]
in $K_0(\cV_K)/(\LL)$.
    
\end{corollary}

\begin{proof}

This follows from Theorem \ref{resolution} and Corollary \ref{classes_coincide}.
    
\end{proof}

\begin{remark}

Equip the free ring $\ZZ[\cV_K]$ with a $\lambda$-structure defined by $\lambda^n(X) = \Sym^n(X)$ for every $K$-variety $X$. Denote by $\ZZ[\cV_K^{\LL\text{-rat}}] \subset \ZZ[\cV_K]$ a subring generated by quasi-projective varieties with $\LL$-rational singularities. Then an equivalent way to state Theorem \ref{resolution} is to say that $\ZZ[\cV_K^{\LL\text{-rat}}]$ is a sub-$\lambda$-ring of $\ZZ[\cV_K]$.

\end{remark}

\vspace{0.3cm}

\section{Irrationality theorem }\label{section:main}

In this section, we prove Theorem \ref{irrationality_result}, building upon the method used in \cite{LL2}. We first recall the construction of motivic measures $\mu_k$ which are used in the proof.

\subsection{The ring \texorpdfstring{$\overline{K}(X)$}{TEXT}.}

Following \cite{LL2}, for a variety $X$ over an algebraically closed field we denote by $\overline{K}(X)$ the quotient of the free abelian group generated by the isomorphism classes $[\cE]$ of vector bundles on $X$ by the subgroup generated by the elements $[\cE] - [\cF] - [\cG]$ such that $\cE$ is isomorphic to $\cF \oplus \cG$ as vector bundles. Note that we do not quotient by general short exact sequences. In particular, the global section functor gives us a well-defined group homomorphism
\[
H^0(X, -): \overline{K}(X) \to K[\Vect].
\]

The tensor product of bundles makes $\overline{K}(X)$ into a ring and the exterior power $\lambda$-operations $\lambda^i([\cE]) \defeq [\Lambda^i \cE]$ define a special $\lambda$-structure on $\overline{K}(X)$ by \cite[Theorem 5.1]{LL2}. The Adams operations $\Psi^n$ are defined on $\overline{K}(X)$ in the usual way. In particular, for any bundle $\cE$ and any $n \ge 0$ the dimension $\dim(H^0(X, \Psi^n \cE)) \in \ZZ$ of the virtual vector space $H^0(X, \Psi^n \cE)$ is defined. Moreover, for any line bundle $\cL$ one has $\Psi^n[\cL] = [\cL^{\otimes n}]$, and $\Psi^n$ are homomorphisms of $\lambda$-rings (since $\overline{K}(X)$ is special).

\subsection{Motivic measures.} \label{motivic_measures}

For a ring $A$, an $A$-valued motivic measure on $\cV_K$ is a ring homomorphism $\mu: K_0(\cV_K) \to A$. For any motivic measure $\mu$ we can define the image of the motivic zeta function of $X$ by
\[
Z_{\mu}(X, t) \defeq 1 + \sum_{n \geqslant 1} \mu([\Sym^n(X)]) t^n \in 1+tA[[t]].
\]
In particular, if $\mu = \Id: K_0(\cV_K) \to K_0(\cV_K)$, then $Z_{\Id}(X, t) = Z(X, t)$.

\begin{definition} \label{def:rationality}
We say that $Z_{\mu}(X, t)$ is \textit{rational} if there exist polynomials $f, g \in A[t]$ such that $g$ is invertible in $A[[t]]$ and $Z_{\mu}(X, t) = f(t)/g(t)$. Clearly, if $Z(X, t)$ is rational, then $Z_{\mu}(X, t)$ is rational for any $\mu$. We also say that $Z(X, t)$ is \textit{pointwise rational} if there exists a motivic measure $\mu: K_0(\cV_K) \to F$, such that $F$ is a field and $Z_{\mu}(X, t)$ is rational.

\end{definition}

The following result by F. Bittner says that in order to construct a motivic measure, it is enough to define it on smooth connected projective varieties and check that the ``blow-up relations'' are satisfied:

\begin{theorem}[{\cite{Bit}}] \label{bittner}

Let $\ZZ[{\cV^{P}_K}]$ be the free abelian group on isomorphism classes of non-singular, connected, projective varieties over $K$. Then the natural homomorphism $\varphi: \ZZ[{\cV^{P}_K}] \to K_0(\cV_K)$ is surjective and the kernel of $\varphi$ is generated by elements $[\varnothing]$ and $([\Bl_Y(X)] - [E]) - ([X]-[Y])$, where $X, Y$ are smooth connected projective varieties, $Y$ is a subvariety of $X$, $\Bl_Y(X)$ is the blow-up of $X$ along $Y$, $E$ is the exceptional divisor of the blow-up.
    
\end{theorem}
\subsection{Measures \texorpdfstring{$\mu_k$}{TEXT}.} \label{Measures} Let $M$ be the multiplicative semigroup of polynomials $f \in \ZZ[s]$ with $f(0) = 1$, and let $\ZZ[M]$ be the semigroup algebra of $M$. Since any element of $M$ can be uniquely written as $g_1^{q_1} \ldots g_r^{q_r}$ where $g_i \in M$ and each $g_i$ is irreducible, it follows that $M$ is a free commutative monoid. In particular, its group completion is a free abelian group and $\ZZ[M]$ is isomorphic to a polynomial ring in infinitely many variables.

The ring $\ZZ[M]$ has a natural $\lambda$-structure defined as follows. Let us identify $\ZZ[s]$ with the ring of isomorphism classes of $\NN$-graded virtual finite-dimensional vector spaces. Then we can view any $m \in M$ as an $\NN$-graded virtual vector space $\sum_{p=0}^N V_p s^p$. We define
\begin{equation*}
\lambda^i(V_p s^p) = 
\begin{cases}
   \Sym^i(V_p)s^{ip}, &  \text{if } p \text{ is even};  \\
   \Lambda^i(V_p)s^{ip}, & \text{if } p \text{ is odd}.
 \end{cases}
 \end{equation*}

Now we assume $K = \overline{K}$. For every smooth connected projective $X$ over $K$ of dimension $d$ and each $k \geqslant 1$ define
\[
\mu_k(X) \defeq 1+ \sum_{i=1}^d h_k^i(X)s^i \in M, \quad h_k^i(X) \defeq \dim H^0(X, \Psi^k \Omega^i_X) \in \ZZ.
\]

\begin{proposition}
The mapping $\mu_k$ induces a well-defined homomorphism $\SB/{K} \to M$ of semigroups. In particular, it induces a motivic measure
\[
\mu_k: K_0(\cV_{{K}}) \to \ZZ[M]
\]
which factors through $K_0(\cV_{{K}})/(\LL) \cong \ZZ[\SB/{K}]$.

\end{proposition}

\begin{proof}

This is proved in \cite[Proposition 6.1]{LL2}. For the reader's convenience, we recall the idea of the proof.

Since any irreducible variety over ${K}$ is birational to a smooth connected projective variety, by Theorem \ref{bittner} it is enough to show that for smooth connected projective varieties:

\begin{enumerate}
    \item[(1.)] $\mu_k(X) = \mu_k(X')$ if $X$ and $X'$ are birational;
    \item[(2.)]  $\mu_k(X \times Y) = \mu_k(X) \mu_n(Y)$;
    \item[(3.)] $\mu_k(\PP^m_{{K}}) = 1$ for $m \geqslant 1$.
\end{enumerate}

The first equality follows from a well-known equality $H^0(X, (\Omega^i_X)^{\otimes n})= H^0(\widetilde{X}, (\Omega^i_{\widetilde{X}})^{\otimes n})$. The second equality is an easy consequence of the K\"unneth theorem and the fact that $\Psi^n$ are $\lambda$-homomorphisms (because $\overline{K}(X)$ is special). Finally, in order to prove (3.), it is enough to show that $\mu_k(\PP^1_{{K}}) = 1$ (by (1.) and (2.)), which is a simple computation.
    
\end{proof}

\subsection{Irrationality result.} 

In this section we assume that $X$ is a smooth connected complex projective variety of dimension $d > 1$. We denote by $\rho_m: Y_m \to \Sym^{m}(X)$ a resolution of singularities of $\Sym^{m}(X)$.




To prove Theorem \ref{irrationality_result}, we will need several lemmas, most of which already appear in \cite{LL2} for the case when $X$ is a surface.

\begin{lemma} \label{inclusion}

For any $m, n \geqslant 0$ there is an inclusion:
\[
H^{0}(Y_m, \, (\Omega^1_{Y_m})^{\otimes n}) \subset H^0(X^m, (\Omega^1_{X^m})^{\otimes n})^{S_m}
\]
    
\end{lemma}

\begin{proof}

Let $U_m = X_{(1^m)} \subset \Sym^{m}(X)$ be an open subset of points with pairwise distinct coordinates and let $Z_m = \Sym^{m}(X) \, \backslash \, U_m$. Since $\dim(X) > 1$, it follows that $U_m$ is the smooth locus of $\Sym^m(X)$ and $\codim(Z_m) = d$. Over $U_m$, the map $\rho_m$ is an isomorphism and the map $\pi: X^m \to \Sym^{m}(X)$ is \'etale, hence
\begin{align*}
  H^{0}(Y_m, \, (\Omega^1_{Y_m})^{\otimes n}) & \subset  H^{0}( \rho_m^{-1}(U_m), \, (\Omega^1_{ \rho_m^{-1}(U_m)})^{\otimes n}) = H^{0}(U_m, \, (\Omega^1_{ U_m})^{\otimes n}) \\
   & = H^{0}(\pi^{-1}(U_m), \, (\Omega^1_{ \pi^{-1}(U_m)})^{\otimes n})^{S_m}
\end{align*}

It remains to note that $\codim(\pi^{-1}(Z_m)) = d > 1$, therefore
\[
H^{0}(\pi^{-1}(U_m), \, (\Omega^1_{ \pi^{-1}(U_m)})^{\otimes n})^{S_m} = H^0(X^m, (\Omega^1_{X^m})^{\otimes n})^{S_m}.
\]

\end{proof}

\begin{lemma}[{\cite[Theorem 1]{AA}}] \label{arapura}
    For any $m, n > 0$ such that $nd$ is even, there is an isomorphism
    \[
    H^0(Y_m, \omega_{Y_m}^{\otimes n}) \cong \Sym^m H^0(X, \omega_X^{\otimes n})
    \]
of vector spaces. 
    
\end{lemma}

\begin{proposition} \label{invariants}

For any $m, i \geqslant 1$ we have
\[
H^0(Y_m, \Omega^i_{Y_m}) = H^0(X^m, \Omega^i_{X^m})^{S_m}.
\]
    
\end{proposition}

We deduce this proposition from the following more general result.

\begin{lemma} \label{klt_singularities}

Let $Z$ be a smooth connected complex quasi-projective variety with an action of a finite group $G$. Let $Y = Z/G$ and let $\pi: \widetilde{Y} \to Y$ be a resolution of singularities. If $Y$ has at most log terminal singularities and the locus $S = \{z \in Z \mid \Stab_G(z) \neq e \} \subset Z$ has codimension greater than one, then
\[
H^0(\widetilde{Y}, \Omega^i_{\widetilde{Y}})  = H^0(Z, \Omega_Z^i)^{G}
\]
for all $i \ge 1$\footnote{In fact, a similar argument shows that $H^0(\widetilde{Y}, \Omega^i_{\widetilde{Y}}) = H^0(Y, \Omega_{Y}^{[i]}) = H^0(Z, \Omega_Z^i)^{G}$, where $\Omega_{Y}^{[i]}$ is a sheaf of reflexive $i$-forms.}.

\end{lemma}

\begin{proof}

Let $U = Z \, \backslash \, S$. 
Consider a commutative diagram
\begin{equation*}
\begin{tikzcd}
\widetilde{Z} \arrow{d}{q'} \arrow{r}{\pi'}  & Z \arrow{d}{q}
\\
\widetilde{Y} \arrow{r}{\pi} & Y
\end{tikzcd}
\end{equation*}
where $q: Z \to Y$ is the quotient map and $\widetilde{Z}$ is a $G$-equivariant resolution of singularities of the fiber product $\widetilde{Y} \times_{Y} Z$. There is a natural map
\[
\Omega^i_{\widetilde{Y}} \to (q'_{\ast} \Omega^i_{\widetilde{Z}})^{G}, 
\]
which is an isomorphism over $\pi^{-1}(U/G) \cong U/G$ (since $U/G$ is smooth and $U \to U/G$ is \'etale). Therefore, we have a map
\[
\varphi: \pi_{\ast} \Omega^i_{\widetilde{Y}} \to \pi_{\ast}((q'_{\ast} \Omega^i_{\widetilde{Z}})^{G}) = ((\pi \circ q')_{\ast}(\Omega^i_{\widetilde{Z}}))^{G} = (q_{\ast} (\pi'_{\ast}\Omega^i_{\widetilde{Z}}))^G = (q_{\ast} \Omega^i_Z)^G
\]
which is an isomorphism over $U/G$. The sheaf $\pi_{\ast} \Omega^{i}_{\widetilde{Y}}$ is coherent and reflexive by \cite[Theorem 1.4]{GKKP}, and the sheaf $(q_{\ast} \Omega^i_Z)^G$ is coherent and reflexive by \cite[Lemmas A.3, A.4]{GKKP}. Since $Y$ is normal and $\codim(S/G) \ge 2$, it follows that $\varphi$ is an isomorphism (\cite[\href{https://stacks.math.columbia.edu/tag/0EBJ}{Tag 0EBJ}]{Stacks}), so
\[
H^0(\widetilde{Y}, \Omega^i_{\widetilde{Y}}) = H^0(Y, \pi_{\ast} \Omega^i_{\widetilde{Y}}) = H^0(Y, (q_{\ast} \Omega^i_Z)^G) = H^0(Z, \Omega^i_Z)^{G}.
\]
    
\end{proof}

\begin{proof}[Proof of Proposition \ref{invariants}.]

By \cite[Proposition 1]{AA}, $\Sym^m(X)$ has canonical singularities for all $m$, so the statement follows from Lemma \ref{klt_singularities}.
    
\end{proof}

\begin{lemma} \label{boundedness}

For any fixed $n$ and $i$ the sequence $(|\dim{H^0(Y_m, \Psi^n \Omega^i_{Y_m})}|)_{m > 0}$ is bounded.
\end{lemma}

\begin{proof}

Since $\Omega^i_{Y_m}$ is a direct summand of $(\Omega^1_{Y_m})^{\otimes i}$ and the Adams operation $\Psi^n$ is a polynomial in the operations $\lambda^1, \ldots, \lambda^n$, it follows that $[\Psi^n \Omega^i_{Y_m}] = \sum_{j=1}^{M} \pm[\cE_j]$ in $\overline{K}(X)$, where $M$ depends only on $n$ and each $\cE_j$ is a direct summand of $(\Omega^1_{Y_m})^{\otimes N}$ for some $N = N(n, i)$. It is therefore enough to prove that $(\dim{H^0(Y_m, (\Omega^1_{Y_m})^{\otimes N})})_{m > 0}$ is bounded. Moreover, by Lemma \ref{inclusion}, we only need to prove the boundedness of
\[
(\dim{H^0(X^m, (\Omega^1_{X^m})^{\otimes N})^{S_m}})_{m > 0}.
\]
This is straightforward and is explained in \cite[Proposition 7.5]{LL2}.

\end{proof}

We will also use the following rationality criterion for a formal powers series with coefficients in $F \defeq \Frac(\ZZ[G])$, where $G$ is a free abelian group.

\begin{lemma}[{\cite[Theorem 2.9]{LL2}}] \label{rationality_of_power_series}

Let $f(t) = \sum_{i \geqslant 0} g_i t^i$ be a power series in $F[[t]]$ with $g_i \in G$ for all $i$. Then $f(t)$ is rational if and only if there exist $p > 0, i_0 \in \NN$, and a sequence $h_1, h_2, h_3, \ldots \in G$ periodic with
period $p$ such that for $i > i_0$ we have $g_{i+p} = h_i g_i$.

\end{lemma}

\begin{proof}[Proof of Theorem \ref{irrationality_result}] Let $F = \Frac{(\ZZ[M])}$ and let ${\mu}'_k: K_0(\cV_{\CC}) \to F$ be the composition of $\mu_k$ with the natural inclusion $\ZZ[M] \to F$. We will prove that there exists some $k$ such that $Z_{{\mu}'_k}(X, t) \in F[[t]]$ is not rational.

By Corollary \ref{class_of_resolution_coincide}, we have $[Y_m] = [\Sym^m(X)]$ in $K_0(\cV_{\CC})/(\LL)$. Thus one has ${\mu}'_k(Y_m) = {\mu}'_k(\Sym^m(X))$ for all $k \ge 1$, so it is enough to prove that
\[
f_{X}(t) \defeq 1+ \sum_{i = 1}^{\infty} {\mu}'_k(Y_m) t^m \in F[[t]]
\]
is not rational for some $k \ge 1$.

We are going to treat parts (a) and (b) separately. First assume that there exists $n \geqslant 1$ such that $H^0(X, \omega_X^{\otimes n}) \neq 0$. Moreover, we can assume that $n$ is even (since if $s$ is a nonzero global section of $\omega_X^{\otimes n}$, then $s \otimes s$ is a nonzero global section of $\omega_X^{\otimes 2n}$). By Lemma \ref{arapura}, we have 
\[
h_n^{md}(Y_m) = \dim{(H^0(Y_m, \omega_{Y_m}^{\otimes n}))} =\dim{(\Sym^m H^0(X, \omega_X^{\otimes n}))} \neq 0,
\]
therefore ${\mu}'_n(Y_m)$ is a polynomial with constant term $1$ and leading term $h_n^{md}(Y_m) s^{md}$.

Now assume that $H^0(X, \Omega_X^{2i}) \neq 0$ for some $i > 0$. By K\"unneth formula, for every $j > 0$ there is an isomorphism
\begin{equation} \label{Kunneth}
H^0(X^m, \Omega^j_{X^m}) \cong \bigoplus_{i_1 + \ldots + i_m = j} H^0(X, \Omega_X^{i_1}) \otimes \cdots \otimes H^0(X, \Omega_X^{i_m}).
\end{equation}

Denote by $\pi_l: X^m \to X$ the projection to the $l$-th factor. Note that the embedding
\[
H^0(X, \Omega^{2i})^{\otimes m} \hookrightarrow H^0(X^m, \Omega^{2im}_{X^m}), \quad s_1 \otimes \cdots \otimes s_m \mapsto \pi^{\ast}_1(s_1) \wedge \cdots \wedge \pi^{\ast}_m(s_m) 
\]
induced from the K\"unneth isomorphism (\ref{Kunneth}) is $S_m$-equivariant because $2i$ is even. Therefore we have an inclusion
\[
\Sym^m H^0(X, \Omega^{2i}) \subset H^0(X^m, \Omega^{2im}_{X^m})^{S_m} = H^0(Y_m, \Omega^{2im}_{Y_m}) = h_1^{2im}(Y_m),
\]
where the middle equality follows from Proposition \ref{invariants}. Since $H^0(X, \Omega^{2i}) \neq 0$ by assumption, it follows that $h_1^{2im}(Y_m) \neq 0$, and hence $\mu'_1(Y_m)$ is a polynomial of degree at least $2im$ and no more than $dm$.

The rest of the argument works for parts (a) and (b) simultaneously. Denote by $G$ a group completion of $M$, and let $g_m \defeq {\mu}'_k(Y_m) \in M$, where we take $k = n$ if $X$ satisfies the condition of part $(a)$, and $k=1$ if $X$ satisfies the condition of part $(b)$. In both cases, we have just proved that $2mC \le \deg(g_m) \le md$ for some constant $C > 0$.

Assume $f_X(t)$ is rational. Then by Lemma \ref{rationality_of_power_series} there exist $p,i_0 \geqslant 1$ and $h \in G$ such that
\[
g_{i_0+jp} = h^jg_{i_0}
\]
for all $j \geqslant 0$. The element $h$ is a rational function in $t$. Since all $g_m$ are non-zero polynomials in $t$ with constant term $1$, it follows that $h$ is also a non-zero polynomial with constant term $1$. Since the degrees of $g_m$ grow, $h$ has degree at least $1$. Let $r \in \ZZ_{>0}$ be the smallest positive integer such that the coefficient of degree $r$ in $h$ is not zero. Then the coefficient of degree $r$ in $g_{i_0+jp}$ is not bounded when $j \to \infty$. This is a contradiction with Lemma \ref{boundedness}.

\end{proof}

\begin{remark}
For part (a) of Theorem \ref{irrationality_result}, it is in fact enough to assume that $X$ is over an arbitrary field $K$ of characteristic zero. Indeed, using that $\Sym^n(X_{\overline{K}}) = \Sym^n(X)_{\overline{K}}$, one reduces the proof to the case $K = \overline{K}$. One then checks that the proofs of all the lemmas which are used for part (a) work over such $K$.
\end{remark}

\begin{remark}

The converse of Theorem \ref{irrationality_result} does not hold if $\dim(X) > 2$ (even if we replace $K_0(\cV_K)$ with $K_0(\cV_K)/(\LL)$). Indeed, if $Y$ is a smooth quintic threefold and $X = Y \times \PP^1$, then $Z(X, t) \equiv Z(Y, t) \; (\mmod{(\LL)})$ is not rational (since $\kappa(Y) = 0$), while Kodaira dimension of $X$ is negative and $H^0(X, \Omega^{2i}_X) = 0$ for $i > 0$.

\end{remark}

\vspace{0.3cm}

\section{Some applications} \label{sec:applications}

In this section we list some applications of the results of previous sections.

For an irreducible variety $X$ over $K$ we denote by $\langle X \rangle$ its class in $\ZZ[\SB/K]$. Note that if $X$ is stably birational to $Y$, then $\Sym^n(X)$ is stably birational to $\Sym^n(Y)$ for all $n > 0$. Indeed, it is enough to show that $\Sym^n(X \times \AA^1)$ is birational to  $\Sym^n(X) \times \AA^n$, which follows from the fact the the projection $\rho_1: \Sym^n(X \times \AA^1) \to \Sym^n(X)$ is a locally trivial fibration over $X_{(1^n)} \subset \Sym^d(X)$ (see Notation \ref{decompositon_of_sym}) with fiber $\AA^n$ by \cite[Lemma 4.4]{Got}.

Hence, we can define a formal power series:
\[
Z_{\SB}(X, t) \defeq \sum_{n \geqslant 0} \langle \Sym^n(X) \rangle t^n \in 1 + t \ZZ[\SB/K] [[t]].
\]

Recall that by Corollary \ref{class_of_resolution_coincide} one has $[\widetilde{\Sym}^n(X)] = [\Sym^n(X)]$ in $K_0(\cV_K)/(\LL)$. Therefore, under the isomorphism $\Phi: K_0(\cV_K)/(\LL) \overset{\sim}{\longrightarrow} \ZZ[\SB/K]$ from Theorem \ref{Lar_Lun} one has
\[\Phi([\Sym^n(X)]) = \langle \Sym^n(X) \rangle
\]
for every smooth connected projective variety $X$ and every $n > 0$. In particular, for such $X$ the image of $Z(X, t)$ in $(K_0(\cV_K)/(\LL))[[t]]$ can be identified with $Z_{\SB}(X, t)$  using the isomorphism $\Phi$.


    


Since measures $\mu_n$ used in the proof of Theorem \ref{irrationality_result} factor through $K_0(\cV_{K})/(\LL)$, the proof also shows the irrationality of $Z_{\SB}(X, t)$ (as a formal powers series with coefficients in $ \ZZ[\SB/K]$) for $X$ as in the statement of the theorem. While it looks almost impossible to us to give a reasonable rationality criterion for $Z(X, t)$ itself in higher dimensions, we hope that one can say more about the rationality of $Z_{\SB}(X, t)$. For example, it is easy to compute $Z_{\SB}(X, t)$ for Severi-Brauer varieties.


    




    

    



\begin{proposition} \label{ratinoalinquotient}

Let $B$ be a Severi-Brauer variety of index $d = i(B)$ over $K$. We have
\begin{equation} \label{zeta_function_modulo_L}
Z_{\SB}(B, t) = \frac{1}{1-t^d} \sum_{i=0}^{d-1} \langle \Sym^{i}(B) \rangle t^i.
\end{equation}
In particular, the image of $Z(B, t)$ in $(K_0(\cV_{K})/(\LL))[[t]]$ is a rational function.

\end{proposition}

\begin{proof}

Denote by $(k, n)$ the greatest common divisor of two integers $k$ and $n$. By \cite[Theorem 1]{Kollar}, for every positive integer $m$ there are stable birational equivalences
\[
\Sym^{m+d}(B) \overset{\text{\text{stab}}}{\sim} \Sym^{(m+d, \, d)}(B) = \Sym^{(m, \, d)}(B) \overset{\text{\text{stab}}}{\sim} \Sym^{m}(B).
\]
Formula (\ref{zeta_function_modulo_L}) now follows immediately.
\end{proof}

More generally, it is natural to ask:
\begin{question}
    Let $X$ be a smooth Fano variety over $K$. Is it true that $Z_{\SB}(X, t)$ is a rational function?
\end{question}

Let us also mention two more applications.

The following is a generalization of Theorem 19 from \cite{Litt} on stable birationality of symmetric powers.

\begin{proposition} \label{sym_are_not_stably_birational}
    
Let $X$ be a smooth connected complex projective variety of dimension $d > 1$.
\begin{itemize}
    \item[(a)] If $\kappa(X) \ge 0$, then $\Sym^m(X)$ is not stably birational to $\Sym^l(X)$ for all $m \neq l$.
    \item[(b)] If $H^0(X, \Omega^{2i}_X) \neq 0$ for some $i > 0$, then $\Sym^m(X)$ is not stably birational to $\Sym^l(X)$ for all $m > ld/(2i)$.
\end{itemize}

\end{proposition}

\begin{proof}

Assume that $\Sym^l(X) \overset{\text{\text{stab}}}{\sim} \Sym^m(X)$ for some $l, m \in \ZZ_{>0}$. Then we have $\widetilde{\Sym}^l(X) \overset{\text{\text{stab}}}{\sim} \widetilde{\Sym}^m(X)$, and hence $[\widetilde{\Sym}^l(X)] = [\widetilde{\Sym}^m(X)]$ in $K_0(\cV_K)/(\LL)$.

Let us take $k$ as in the proof of Theorem \ref{irrationality_result}, and let $g_s \defeq \mu_k(\widetilde{\Sym}^s(X))$. Then, if the assumptions of the proposition hold, the polynomials $g_m$ and $g_l$ have different degrees (which was shown in the course of the proof of Theorem \ref{irrationality_result}). In particular, $g_m \neq g_l$, and so $[\widetilde{\Sym}^m(X)] \neq [\widetilde{\Sym}^l(X)]$ in $K_0(\cV_K)/(\LL)$, which is a contradiction.
\end{proof}

The following proposition was already stated in \cite[Question 6.7]{LL2}.

\begin{proposition} \label{hom_of_lambda}
    
The motivic measure
\[
\mu_1: K_0(\cV_{\CC}) \to \ZZ[M]
\]
is a homomorphism of $\lambda$-rings.

\end{proposition}

\begin{lemma}\cite[Lemma 6.5]{LL2} \label{lemma_for_hom_of_lambda}
Let $Y$ be a smooth projective variety over $K$ and let $\cF$ be a vector bundle on $Y$ of rank $r$. For a fixed $m$ denote by $\pi_l: X^m \to X$ the projection to the $l$-th factor and put $\cF^{\boxplus m} \defeq \pi_1^{\ast}\cF \oplus \cdots \oplus \pi_m^{\ast}\cF$. There is a natural isomorphism of graded vector spaces:
\[
\sum_j H^0(Y^m, \Lambda^{j}(\cF^{\boxplus m}))^{S_m} t^j = \lambda^m\Bigl(\sum_j H^0(Y, \Lambda^{j}(\cF)) t^j \Bigr).
\]
    
\end{lemma}

\begin{proof}[Proof of Proposition \ref{hom_of_lambda}.] By Proposition \ref{invariants}, for every smooth connected complex projective variety $X$ we have $H^0(Y_m, \Omega^i_{Y_m}) = H^0(X^m, \Omega^i_{X^m})^{S_m}$ for all $m, i \ge 1$\footnote{Technically, in Proposition \ref{invariants} we assumed that $\dim(X) > 1$, but it is straightforward to see that the equality also holds if $\dim(X) \le 1$.}. Moreover, $\mu_1(\lambda^m(X)) = \mu_1(Y_m)$ by Corollary \ref{class_of_resolution_coincide}, so we need to prove that
\[
1 + \sum_{i=1}^d H^0(X^m, \Omega^i_{X^m})^{S_m} t^i = \lambda^m \Bigl(1 +  \sum_{i=1}^d H^0(X, \Omega^i_X) t^i\Bigr)
\]
for all $m \ge 1$. This follows from Lemma \ref{lemma_for_hom_of_lambda} applied for  $Y = X$ and $\cF = \Omega^1_X$.
    
\end{proof}

\end{document}